\documentclass{amsart}
\usepackage{xcolor}
\usepackage{mathrsfs}
\usepackage{amsmath,amsthm,amsopn,amstext,amscd,amsfonts,amssymb,mathrsfs,mathtools}
\usepackage{dsfont}
\usepackage{comment}
\usepackage{xcolor}
\usepackage{float}
\usepackage[active]{srcltx}
\usepackage{graphicx, mathdots}

  \usepackage{mathtools}

\newtheorem{definition}{\sc Definition}[section]
\newtheorem{theorem}{\sc Theorem}[section]
\newtheorem{conj}{\sc Conjecture}[section]
\newtheorem{lemma}{\sc Lemma }[section]

\newtheorem{coro}{\sc Corollary}[section]
\newtheorem{proposition}{\sc Proposition}[section]
\newtheorem{obs}{\sc Remark}[section]
\usepackage[active]{srcltx}

\newcommand{\dps}{\displaystyle}
\usepackage{graphicx, epsfig, subfig}
\usepackage{lscape}
\usepackage{rotating}
\usepackage{caption}
\usepackage{multicol}
\usepackage{colortbl}
\usepackage{nicematrix}
\newcommand*\pPq[5]{%
 \begingroup
 \begingroup\lccode`~=`,
   \lowercase{\endgroup\def~}{\pFcomma\mkern\pFqskip}%
 \mathcode`,=\string"8000
 {}_{#1}\phi_{#2}\left(\left.\genfrac..{0pt}{}{#3}{#4}\right|\,#5\right)%
 \endgroup
}
\newcommand{\prodint}[1]{\left\langle {#1}\right\rangle}

\usepackage{hyperref}
\newmuskip\pFqskip
\mathchardef\pFcomma=\mathcode`, 

\makeatletter
\def\BState{\State\hskip-\ALG@thistlm}
\makeatother

\def\downbar#1{
\setbox10=\hbox{$#1$}
            \dimen10=\ht10 \advance\dimen10 by 2.5pt
            \ifdim \dimen10<15pt 
               \advance\dimen10 by -0.5pt
               \dimen11=\dimen10
               \advance\dimen10 by 2.5pt
               \lower \dimen11
            \else \lower \ht10 \fi
            \hbox {\hskip 1.5pt \vrule height \dimen10 depth \dp10}}
\def\upbar#1{
\setbox10=\hbox{$#1$}
            \dimen10=\ht10 \advance\dimen10 by \dp10 \advance\dimen10 by 2.5pt
            \ifdim \dimen10<15pt 
                \advance\dimen10 by 2pt \fi
            \raise 2.5pt \hbox {\hskip -1.5pt \vrule height \dimen10}}

\begin{document}
\title[Epilegomena to the study of semiclassical orthogonal polynomials]{Epilegomena to the study of semiclassical orthogonal polynomials}
\author{K. Castillo}
\address{CMUC, Department of Mathematics, University of Coimbra, 3001-501 Coimbra, Portugal}
\email{ kenier@mat.uc.pt}
\author{D. Mbouna}
\address{Department of Mathematics, Faculty of Sciences, University of Porto, Campo Alegre st., 687, 4169-007 Porto, Portugal}
\email{dieudonne.mbouna@fc.up.pt}
\date{\today}
\begin{abstract}
In his monograph [Classical and quantum orthogonal polynomials in one variable, Cambridge University Press, 2005 (paperback edition 2009)], Ismail conjectured that certain structure relations involving the Askey-Wilson operator characterize proper subsets of the set of all $\mathcal{D}_q$-classical orthogonal polynomials, here to be understood as the
Askey-Wilson polynomials and their limit cases. In this paper we give two characterization theorems for $\mathcal{D}_q$-semiclassical (and classical) orthogonal polynomials in consonance with the pioneering works by Maroni [Ann. Mat. Pura. Appl. (1987)] and Bonan, Lubinsky, and Nevai [SIAM J. Math. Anal. 18 (1987)] for the standard derivative, re-establishing in this context  the perfect ``symmetry" between the standard derivative and the Askey-Wilson operator. As an application, we present a sequence of  $\mathcal{D}_q$-semiclassical orthogonal polynomials of class two that disproves Ismail's conjectures. Further results are presented for Hahn's operator.
 \end{abstract}
\maketitle
\section{Introduction}
The term semiclassical for a special class of orthogonal polynomials was coined in 1984 by Hendriksen and van Rossum (see \cite{HR85}) during the Laguerre Symposium held at Bar-le-Duc, whose main speaker was Dieudonné. However, in the first line of his monumental work entitled {\em ``Une th\'eorie alg\'ebrique des polyn\^omes orthogonaux. Applications aux polyn\^omes orthogonaux semi-classiques''} (see \cite{M91}), Maroni referred to Shohat\footnote{On his recent visit to Portugal, C. Brezinski showed us the only known photo of J. Shohat, which will appear in his book on the history of orthogonal polynomials written in collaboration with M. Redivo-Zaglia.} (see \cite{S39}) as {\em ``l'inventeur des polyn\^omes semi-classiques''}. It was at the hands of Maroni that semiclassical orthogonal polynomials have become such a highly developed topic, although, as he himself points out, these sequences of orthogonal polynomials (OP) have always been present in certain structure relations which are as old as orthogonal polynomials themselves. Let us recall one of the best known problems in this regard. According to Al-Salam and Chihara (see \cite[p. 69]{AC72}), Askey raised the question of characterizing OP, $(P_n)_{n\geq 0}$, satisfying 
\begin{align}\label{Askey}
\phi\,P'_n=\sum_{j=-M}^N a_{n, j} P_{n+j} \qquad (c_{n,j}\in \mathbb{C};\, M, N \in \mathbb{N}),
\end{align}
$\phi$ being a polynomial which does not depend on $n$. In \cite{AC72} it was proved that the only OP that satisfy \eqref{Askey} for $M=N=1$ are the old classical polynomials, i.e., Jacobi, Bessel, Hermite, and Laguerre polynomials. However, for arbitrary values of $M$ and $N$, there are nonclassical OP for which \eqref{Askey} holds. For instance, in 1984 Koornwinder considered (see \cite{K84}), without explicit mention of this, OP, $(P_n^{\alpha,\beta,M_0,M_1})_{n\geq0}$, satisfying \eqref{Askey}, whose weight function is a linear combination of the Jacobi measure and two mass points at $-1$ and $1$: $(1-x)^\alpha(1+x)^\beta+M_0\delta(x+1)+M_1\delta(x-1)$ ($\alpha, \beta>-1; M_0, M_1\in \mathbb{N}$). Indeed,  as Koornwinder recently proved in a private communication (see also \cite{K23}),
\begin{align*}
&\int_{-1}^1 (1-x^2)^2\, \frac{\mathrm{d} P_n^{\alpha,\beta,M_0,M_1}}{\mathrm{d}x}(x) Q(x) (1-x)^\alpha (1+x)^\beta \mathrm{d}x\\[7pt]
&\quad +M_0 (1-x^2)^2\,\left.\frac{\mathrm{d} P_n^{\alpha,\beta,M_0,M_1}}{\mathrm{d}x}(x) Q(x)\right|_{x=1}\\[7pt]
&\quad+M_1 (1-x^2)^2 \,\left.\frac{\mathrm{d} P_n^{\alpha,\beta,M_0,M_1}}{\mathrm{d}x}(x) Q(x)\right|_{x=-1}=0,
\end{align*}
$Q$ being a polynomial of degree at most $n-4$. Hence
$$
(1-x^2)^2 \, \frac{\mathrm{d} P_n^{\alpha,\beta,M_0,M_1}}{\mathrm{d}x}(x)=\sum_{j=-3}^{3} c_{n,j} P_{n+j}^{\alpha,\beta,M_0,M_1}(x)\qquad (c_{n,j}\in \mathbb{C}).
$$
In 1985 (see \cite[Theorem 3.1]{M85} and \cite[Theorem 3.1]{M87}), Maroni, in his {\em prolegomena} to the study of semiclassical orthogonal polynomials, answers Askey's question by characterizing semiclassial OP\footnote{According to Maroni (see \cite[Definition 7.1]{M91}), the classical orthogonal polynomials are semiclassical of class $0$. However, subsequently, when we refer to semiclassical orthogonal polynomials, we exclude the classical ones.} (see also \cite{M99}). (Note that from \cite[Theorem 3.1]{M85} and \cite[Theorem 3.1]{M90} (see also \cite{KP97}), it follows immediately that Koornwinder's polynomials are semiclassical OP satisfying \eqref{Askey}.) In the same year, Bonan, Lubinsky, and Nevai solved a more general problem using analytic methods in the framework of orthogonality in the positive definite sense (see \cite[Theorem 1.1]{BLN87}).  Although much has been written about this, there are still outstanding problems that remain open. For instance, we may rewrite Askey's question by changing the standard derivative by the Askey-Wilson operator. Recall that usually the Askey-Wilson divided difference operator $\mathcal{D}_q:\mathcal{P}\to\mathcal{P}$ is defined by
\begin{align}\label{0.3}
\mathcal{D}_q\,f(x)=\frac{\breve{f}\left(q^{1/2} e^{i\theta}\right)
-\breve{f}\big(q^{-1/2} e^{i\theta}\big)}{\breve{e}\big(q^{1/2}e^{i\theta}\big)-\breve{e}\big(q^{-1/2} e^{i\theta}\big)},
\end{align}
where $\breve{f}(e^{i\theta})=f(\cos \theta)$ for each polynomial $f$, $e(x)=x$, and $\theta$ is not necessarily a real number (see \cite[Section 12.1]{I05}).  Here and subsequently, $0<q<1$ is assumed fixed. ($\mathcal{P}$ denotes the vector space of all polynomials  with complex coefficients.) Taking $e^{i\theta}=q^s$ in \eqref{0.3}, $\mathcal{D}_q$ reads
$$
\mathcal{D}_q\,f(x(s))=\frac{f(x(s+1/2))-f(x(s-1/2))}{x(s+1/2)-x(s-1/2)},
$$
with $x(s)=(q^s+q^{-s})/2$.
What is expected by most of the people working on the subject, in contradiction with the case of the standard derivative operator, is that the answer to this and other related questions be proper subsets of the set of all $\mathcal{D}_q$-classical orthogonal polynomials, here to be understood as the
Askey-Wilson polynomials and their limit cases. Recall that the Askey-Wilson polynomials (see \cite[Section 14.1]{K})
$$
p_n(x; a, b, c, d\,|\, q)=a^{-n}\,(ab, ac, ad; q)_n\, \pPq{4}{3}{q^{-n}, a b c b d q^{n-1}, a e^{i \theta},  a e^{-i \theta}}{ab, ac, ad}{q, q},
$$
where $x=\cos \theta$, are the $q$-analogues of the Wilson polynomials. (If we take $a=q^{1/2\alpha+1/4}$, $b=q^{1/2\alpha+3/4}$, $c=-a$, and $d=-b$, we get the continuous $q$-Jacobi polynomials. If we take $c=d=0$, we get the Al-Salam-Chihara polynomials.) In this sense, there are in the literature two well-known conjectures posed by Ismail (see \cite[Conjecture 24.7.8]{I05} and \cite[Conjecture 24.7.9]{I05}). 

\begin{conj}\label{conj1}
Let $(P_n)_{n\geq0}$ be a sequence of orthogonal polynomials and let $\phi$ be a polynomial which does not depend on $n$. If
\begin{align}\label{0.2Dq}
\phi\,\mathcal{D}_q\,P_n=\sum_{j=-1}^1 a_{n, j} P_{n+j}  \qquad (a_{n,j}\in \mathbb{C}),
\end{align}
then $P_n$ is a multiple of the continuous $q$-Jacobi polynomials or Al-Salam-Chihara polynomials, or special or limiting cases of them. The same conclusion holds if 
\begin{align}\label{0.2Dq-general}
\phi\,\mathcal{D}_q\,P_n=\sum_{j=-M}^N a_{n, j} P_{n+j}\qquad (a_{n,j}\in \mathbb{C};\, M, N \in \mathbb{N}).
\end{align}
\end{conj}

\begin{conj}\label{conj2}
Let $(P_n)_{n\geq 0}$ be a sequence of orthogonal polynomials and $\phi$ be a polynomial of degree at most $4$. Then $(P_n)_{n\geq 0}$ satisfies
\begin{align}\label{second-conject-equat}
\phi\,\mathcal{D}_q ^2 P_n=\sum_{j=-M} ^N a_{n,j}P_{n+j} \qquad (a_{n,j}\in \mathbb{C};\, M, N \in \mathbb{N}),
\end{align}
if and only if $P_n$ is a multiple of $p_n(x; a, b, c, d\,|\, q)$ for some for some parameters $a, b, c, d$.
\end{conj}

 In \cite{A95}, Al-Salam proved Conjecture \ref{conj1} for $\phi=1$ by characterizing the continuous $q$-Hermite polynomials (see \cite[Section 14.26]{K}). In \cite{CMP22b}, we prove that the Al-Salam Chihara polynomials (see \cite[Section 14.26]{K}), with nonzero parameters $a$ and $b$ such that $a/b=q^{\pm 1/2}$, are the only OP satisfying \eqref{0.2Dq} for $\deg \phi=1$. We also prove that the Chebyschev polynomials of the first kind and the continuous $q$-Jacobi polynomials (see \cite[Section 14.10]{K}) are the only ones satisfying \eqref{0.2Dq} for $\deg \phi=2$.  Moreover, in \cite[Proposition 2.1]{CM22} we prove that the continuous dual $q$-Hahn polynomials (see \cite[Section 14.3]{K}), with parameters $a=1$, $b=-1$, $c=q^{1/4}$, and $q$ replaced by $q^{1/2}$, satisfy \eqref{0.2Dq-general} with $M=2$ and $N=1$, which  disproves the second part of  Conjecture \ref{conj1}.  On the other hand, Conjecture \ref{conj2} is claimed to be positively solved in \cite{KJconj}, but the authors only proved, partially, the case $M=2$ and $N=2$ \footnote{It is worth pointing out that the main result of \cite{KJconj} is a particular case of a direct consequence of Sonine-Hahn's theorem for the Askey-Wilson operator (see remark after \cite[Theorem 1.2]{CM23}).}. Recently we noted that above conjectures are related with the theory of $\mathcal{D}_q$-semiclassical orthogonal polynomials. Indeed, in Section \ref{Sec3}, we  characterize $\mathcal{D}_q$-semiclassical (and classical) orthogonal polynomials from the structure relation \eqref{0.2Dq-general},  in consonance with the  works by Maroni \cite{M87} and Bonan, Lubinsky, and Nevai \cite{BLN87} for the standard derivative,  re-establishing in this context  the perfect ``symmetry" between the standard derivative and the Askey-Wilson operator. As an application of these results, in Section \ref{Sec4}, we present an example of $\mathcal{D}_q$-semiclassical orthogonal polynomials that disproves Conjecture \ref{conj1}. In Section \ref{Sec4} we also show that the OP that disproves Conjecture \ref{conj1} also disproves Conjecture \ref{conj2}. Finally,  in Section \ref{Sec6}, we explore our ideas when instead of the Askey-Wilson operator we consider the Hahn operator, which is related with the structure relation of another conjecture posed by Ismail (see \cite[Conjecture 24.7.7]{I05}), but first some preliminary definitions and basic results are needed.

\section{Preliminary results}\label{Sec2} 
Let $\mathcal{P}^*$ be the set of all linear forms on $\mathcal{P}$ and let $\mathcal{P}_n$ be the subspace of $\mathcal{P}$ of all polynomials with degree less than or equal to $n$. Set $\mathcal{P}_{-1}=\{0\}$. A free system in $\mathcal{P}$ is a sequence $(Q_n)_{n\geq0}$ such that
$Q_n\in\mathcal{P}_n\setminus\mathcal{P}_{n-1}$ for each $n$.
A free system $(P_n)_{n\geq0}$ is called OP
with respect to ${\bf u}\in\mathcal{P}^*$ if 
$$
\langle{\bf u},P_nP_m\rangle=h_n\delta_{n,m}\quad(m=0,1,\ldots; \,h_n\in\mathbb{C}\setminus\{0\}),
$$
$\langle{\bf u},f\rangle$ being the action of ${\bf u}$ on $f$. ${\bf u}$ is called regular if there exists an OP with respect to it. Recall that a (monic) OP, $(P_n)_{n\geq 0}$, satisfies the following recurrence relation:
\begin{align}
x P_n(x)=P_{n+1}(x)+B_nP_n(x)+C_nP_{n-1}(x)\qquad  (B_n\in \mathbb{C},\, C_{n+1} \in \mathbb{C}\setminus \{0\}), \label{TTRR}
\end{align}
with initial conditions $P_{-1}=0$ and $P_{0}=1$. Hence it follows that
$$
B_n=\frac{\langle \mathbf{u}, x P^2_n(x)\rangle}{\langle \mathbf{u}, P^2_{n}(x)\rangle}, \qquad C_n=\frac{\langle \mathbf{u}, P^2_n\rangle}{\langle \mathbf{u}, P^2_{n-1}\rangle}.
$$
(Of course, there is no loss of generality in assuming $C_0=0$.) Since the elements of $\mathcal{P}^*$ are completely determined by its action on a system of generators of  $\mathcal{P}$, we say that $\mathbf{u}=\mathbf{v}$ if and only if 
$$
\mathbf{u}_n=\prodint{\mathbf{u}, x^n}=\prodint{\mathbf{v}, x^n}
$$ for all $n\in \mathbb{N}$. In the set $\mathcal{P}^*$, addition and multiplications by scalars can be defined by
\begin{align*}
\prodint{\mathbf{u}+\mathbf{v}, x^n}&=\prodint{\mathbf{u}, x^n}+\prodint{\mathbf{v}, x^n},\\[7pt]
\prodint{c \mathbf{u}, x^n}&=c \prodint{\mathbf{u}, x^n}\qquad (c\in \mathbb{C}),
\end{align*}
for all $n\in \mathbb{N}$. $\mathcal{P}^*$, endowed with these operations, is a vector space over $\mathcal{P}$.  In $\mathcal{P}^*$, the identity for the additivity is denoted by $\mathbf{0}$ and called the zero. The zero is therefore defined by the relation $\prodint{\mathbf{0}, x^n}=0$ for all $n \in \mathbb{N}$. Note that $f\,\mathbf{u}=g\,\mathbf{u}=\mathbf{0}$ $(f,g \in \mathcal{P})$ if and only if $\mathbf{u}=\mathbf{0}$.
The left multiplication of ${\bf u}$ by $f\in \mathcal{P}$, denoted by $f{\bf u}:\mathcal{P}\to\mathcal{P}$,
is the form defined by
$$
\langle f {\bf u}, x^n\rangle=\langle {\bf u}, f x^n\rangle,
$$
for all $n\in \mathbb{N}$. The division of $\mathbf{u}$ by a polynomial, denoted by $(x-c)^{-1}\mathbf{u}:\mathcal{P}\to\mathcal{P}$, is the form defined by
$$
\prodint{(x-c)^{-1}\mathbf{u}, f}=\prodint{\mathbf{u}, \frac{f(x)-f(c)}{x-c}}\qquad (c\in \mathbb{C};\, f\in \mathcal{P}).
$$
Define also  $\mathbf{\delta}_c: \mathcal{P} \to \mathbb{C}$ by $\mathbf{\delta}_c f(x)=f(c)$. We check at once that
$$
(x-c)((x-c)^{-1}\mathbf{u})=\mathbf{u}, \qquad (x-c)^{-1}((x-c)\mathbf{u})=\mathbf{u}-\mathbf{u}_0\,\delta_c.
$$

 $\mathcal{P}$ may be endowed with an appropriate strict inductive limit topology such that the algebraic and
the topological dual spaces of $\mathcal{P}$ coincide (see \cite[Chapter 13]{T}), that is, 
\begin{align}\label{pascal}
\mathcal{P}^*=\mathcal{P}'.
\end{align}
Given a free system $(Q_n)_{n\geq0}$, the corresponding dual basis is a sequence of linear forms
${\bf a}_n:\mathcal{P}\to\mathbb{C}$ such that
$$
\langle{\bf a}_n,Q_m\rangle=\delta_{n,m},
$$
and so, for an OP $(P_n)_{n\geq0}$, ${\bf a}_n$ is explicitly given by
$$
{\bf a}_n=\frac{P_n}{\langle{\bf a}_n,P^2_n\rangle} \mathbf{u}.
$$
In the sense of the weak dual topology, it would be easy to explicitly build bases in dual space. Indeed,
$$
\mathbf{u}=\sum_{j=0}^\infty \prodint{\mathbf{u}, Q_j} \mathbf{a}_j\qquad (\mathbf{u} \in \mathcal{P}');
$$
this will be essential in the sequel. For more details we refer the reader to \cite{M91} (see also \cite{CP}). 

The Askey-Wilson  average operator $\mathcal{S}_q:\mathcal{P}\to\mathcal{P}$ is defined by  (see \cite[p. 301]{I05}) 
\begin{align*}
\mathcal{S}_q f(x(s))=\frac{f\big(x(s+1/2)\big)+f\big(x(s-1/2)\big)}{2}.
\end{align*}
 for every polynomial $f$. It is easy to see that $\mathcal{D}_q\,x^n =\gamma_n x^{n-1}+ (\text{lower degree terms})$ and $\mathcal{S}_q\,x^n=\alpha_n x^n+ (\text{lower degree terms})$ for all $n\in \mathbb{N}$, where we have set
\begin{align}\label{ag}
\alpha_n&= \frac{q^{n/2} +q^{-n/2}}{2},\qquad  \gamma_n = \frac{q^{n/2}-q^{-n/2}}{q^{1/2}-q^{-1/2}}.
\end{align}
Set $\gamma_{-1}=-1$ and $\alpha_{-1}=\alpha$.  For every ${\bf u} \in \mathcal{P^*}$ and $f\in \mathcal{P}$, $\mathbf{D}_q:\mathcal{P}^*\to\mathcal{P}^*$ and $\mathbf{S}_q:\mathcal{P}^*\to\mathcal{P}^*$ are defined by transposition:
\begin{align*}
\langle \mathbf{D}_q{\bf u},f\rangle=-\langle {\bf u},\mathcal{D}_q f\rangle,\qquad  \langle\mathbf{S}_q{\bf u},f\rangle=\langle {\bf u},\mathcal{S}_q f\rangle.
\end{align*}  
The next definition extends the definition of classical linear forms given by Geronimus \cite{G40} and Maroni (see \cite[Proposition 2.1]{M93}).

\begin{definition}\label{NUL-def}  \cite[Definition 3.1]{CMP22a}
${\bf u}\in\mathcal{P}^*$ is called $\mathbf{D}_q$-classical if it is regular and there exist 
$\phi\in\mathcal{P}_2\setminus \mathcal{P}_{-1}$ and $\psi\in\mathcal{P}_1\setminus \mathcal{P}_{-1}$
such that
\begin{align}\label{NUL-Pearson}
\mathbf{D}_q (\phi{\bf u})=\mathbf{S}_q(\psi{\bf u}).
\end{align}
$($We will call it simply classical when no confusion can arise.$)$
\end{definition}

Observe that \eqref{NUL-Pearson} condenses all the information of a sequence of classical orthogonal polynomials in the first three non-constant polynomials of said sequence. The next theorem gives tractable necessary and sufficient conditions for the existence of solutions of \eqref{NUL-Pearson}, characterizing the linear form $\mathbf{u}$ and, in particular, solving the question of the existence of classical OP.

\begin{theorem}\label{regAW}\cite[Theorem 4.1]{CMP22a}
Suppose that ${\bf u} \in \mathcal{P}^*$ satisfies \eqref{NUL-Pearson} with $\phi(x)=ax^2+bx+c$ and $\psi(x)=dx+e$. Then ${\bf u}$ is regular if and only if 
\begin{align*}
 d_n\neq 0,\qquad \phi ^{[n]} \left(-\frac{e_n}{d_{2n}}\right)\neq 0,
\end{align*}
for all $n\in \mathbb{N}$, where $d_n=a\gamma_n+d\alpha_n$, $e_n=b\gamma_n+e\alpha_n$, $\alpha_n$ and $\gamma_n$ being defined by \eqref{ag}, and
 \begin{align*}
\phi^{[n]}(x)=\big(d(\alpha^2-1)\gamma_{2n}+a\alpha_{2n}\big)
\big(x^2-1/2\big)+\big(b\alpha_n+e(\alpha^2-1)\gamma_n\big)x+ c+a/2.
\end{align*}
\end{theorem}

OP with respect to ($\mathbf{D}_q$-)classical linear forms are called ($\mathcal{D}_q$-)classical polynomials.  Unlike when dealing with the standard derivative, in the case of Definition \ref{NUL-def} it is still an open problem to describe the solutions of \eqref{NUL-Pearson} (see \cite[Theorem 3.2]{CP}).
\begin{theorem}\label{T}\cite[Theorem 20.1.3]{I05}
The equation
\begin{align}\label{Ismail}
f(x)\mathcal{D}^2_q\, y+g(x) \mathcal{S}_q \mathcal{D}_q \, y+h(x)\, y=\lambda_n\, y
\end{align}
has a polynomial solution $P_n\in \mathcal{P}_n\setminus\mathcal{P}_{n-1}$ if and only if $P_n$ is a multiple of $p_n(x; a, b, c, d\,|\, q)$ for some parameters $a, b, c, d$ including limiting cases as one or more of the parameters tends to $\infty$. In all these cases $f$, $g$, $h$, and $\lambda_n$ reduce to
\begin{align*}
f(x)&=-q^{-1/2}(2(1+\sigma_4)x^2-(\sigma_1+\sigma_3)x-1+\sigma_2-\sigma_4),\\[7pt]
g(x)&=\frac{2}{1-q} (2(\sigma_4-1)x+\sigma_1-\sigma_3), \qquad h(x)=0,\\[7pt]
\lambda_n&=\frac{4 q(1-q^{-n})(1-\sigma_4 q^{n-1})}{(1-q)^2}, 
\end{align*}
or a special or limiting case of it, $\sigma_j$ being the jth elementary symmetric function of the Askey-Wilson parameters. 
\end{theorem}

Let us recall some useful operations.

\begin{lemma}\label{lema}\cite[Lemma 2.1]{CMP22a}
Let $f,g\in\mathcal{P}$ and ${\bf u}\in\mathcal{P}^*$. Then the following hold:
\begin{align}
\mathcal{D}_q \big(fg\big)&= \big(\mathcal{D}_q f\big)\big(\mathcal{S}_q g\big)+\big(\mathcal{S}_q f\big)\big(\mathcal{D}_q g\big), \label{def-Dx-fg} \\[7pt]
\mathcal{S}_q ( fg)&=(\mathcal{D}_q f) (\mathcal{D}_q g)\texttt{U}_2  +(\mathcal{S}_q f\big) \big(\mathcal{S}_q g), \label{def-Sx-fg} \\[7pt]
f\mathcal{D}_q g&=\mathcal{D}_q\left((\mathcal{S}_qf-\alpha^{-1}\texttt{U}_1\mathcal{D}_qf )g\right)-\alpha^{-1}\mathcal{S}_q (g\mathcal{D}_q f), \label{def-fDxg} \\[7pt]
\alpha{\bf D}_q(f{\bf u})&=\left(\alpha\mathcal{S}_qf-\texttt{U}_1\mathcal{D}_qf \right){\bf D}_q{\bf u}+\mathcal{D}_qf{\bf S}_q{\bf u},\label{def-Dx-fu}   \\[7pt]
\alpha {\bf S}_q (f{\bf u}) &=(\alpha ^2 \texttt{U}_2 -\texttt{U}_1 ^2)\mathcal{D}_qf~{\bf D}_q  {\bf u}+(\alpha \mathcal{S}_qf+\texttt{U}_1\mathcal{D}_qf ){\bf S}_q {\bf u},\label{def-Sx-fu}\\[7pt]
f{\bf D}_q {\bf u}&={\bf D}_q\left(\mathcal{S}_qf~{\bf u}  \right)-{\bf S}_q\left(\mathcal{D}_qf~{\bf u}  \right), \label{def-fD_x-u}\\[7pt]
f{\bf S}_q {\bf u}&={\bf S}_q\left(\mathcal{S}_qf~{\bf u}  \right)-{\bf D}_q\left(\texttt{U}_2 \mathcal{D}_qf~{\bf u}  \right), \label{def-fS_x-u}\\[7pt]     
\alpha \mathbf{D}_q ^n \mathbf{S}_q {\bf u}&= \alpha_{n+1} \mathbf{S}_q \mathbf{D}_q^n {\bf u}
+\gamma_n \texttt{U}_1\mathbf{D}_q^{n+1}{\bf u} \qquad (n \in \mathbb{N}),\label{DxnSx-u} 
\end{align}
where $\texttt{U}_1 (x)=(\alpha ^2 -1)x$ and $\texttt{U}_2 (x)=(\alpha ^2 -1)(x^2-1).$
\end{lemma}

The following result clarifies which are the $\mathcal{D}_q$-classical orthogonal polynomials.

\begin{proposition}\label{Prop}
The $\mathcal{D}_q$-classical sequences of orthogonal polynomials are the sequences of Askey-Wilson polynomials $(p_n(x; a, b, c, d\,|\, q))_{n=0}^\infty$  for some parameters $a, b, c, d$ including limiting cases as one or more of the parameters tends to $\infty$.
\end{proposition}
\begin{proof}
This follows from Theorem \ref{T}, after showing the equivalence between \eqref{NUL-Pearson} and \eqref{Ismail} with $h=0$ (see \cite[Theorem 5]{FKN11}). (The interested reader can also prove this easily following the proof of \cite[Proposition 2.8]{M93}.)
\end{proof}



From Definition \ref{NUL-def} we introduce $\mathbf{D}_q$-semiclassical linear forms in a natural way (see, for instance, \cite[Section 3]{M93}).

\begin{definition}\label{semi}
We call a linear form, in $\mathcal{P}^*$, $\mathbf{D}_q$-semiclassical if it is regular, not $\mathbf{D}_q$-classical, and there exist
two polynomials $\phi$ and $\psi$ with at least one of them nonzero, such that \eqref{NUL-Pearson} holds. $($We will call it simply semiclassical when no confusion can arise.$)$
\end{definition}
Under the conditions of Definition \ref{semi}, necessarily both $\phi$ and $\psi$ are not zero and $\deg \psi\geq1$. The class of ${\bf u}$ is the positive integer
\begin{align*}
s=\min_{(\phi,\psi)\in\mathcal{P}_{\bf u}}\max\{\mbox{\rm deg}\,\phi-2,\mbox{\rm deg}\,\psi-1\},
\end{align*}
where $\mathcal{P}_{\bf u}$ is the set of all pairs $(\phi,\psi)$ of nonzero polynomials such that \eqref{NUL-Pearson} holds. (Note that when $s=0$ we have the classical linear forms.)
The pair $(\phi,\psi)\in\mathcal{P}_{\bf u}$ where the class of ${\bf u}$ is attained is unique up to a constant factor. OP with respect to a ($\mathbf{D}_q$-)semiclassical form of class $s$ are called a ($\mathcal{D}_q$-)semiclassical orthogonal polynomials of class $s$.

We end this section with the following definition. 

\begin{definition}\label{adm}
We call a pair of polynomials $(\phi, \psi)$, $\phi(x)=a_p\,x^p+ (\text{lower degree}$ $\text{terms})$ and $\psi(x)=b_q\,x^q+ (\text{lower degree terms})$ $(p\in \mathbb{N}, q\in \mathbb{N}\setminus\{0\})$, admissible if $p-1\not=q$ or $a_p\,\gamma_{n}+b_q\,\alpha_{n-1}\not=0$ whenever $p-1=q$.
\end{definition}

\begin{obs}
We emphasize that from the point of view of \cite{CMP22a}, in this paper we are working with the particular lattice $x(s)=(q^{-s}+q^{s})/2$. If we consider the lattice $x(s)=\mathfrak{c}_6$, in the notation of  \cite{CMP22a}, we have $\alpha_{n-1}=1$ and $\gamma_n=n$, and so Definition \ref{adm} reduces the admissibility condition for the standard derivative $($see \cite[p. 119]{M91} and \cite[p. 46]{P93}$)$.
\end{obs}

\section{Characterization theorems}\label{Sec3}

The next theorem characterize $\mathcal{D}_q$-classical and $\mathcal{D}_q$-semiclassical orthogonal polynomials from the structure relation \eqref{0.2Dq-general}. 

\begin{theorem}\label{main-theorem}
Let ${\bf u} \in \mathcal{P}'$ be regular and let $(P_n)_{n\geq 0}$  denote the corresponding sequence of orthogonal polynomials.  The following conditions are equivalent:
\begin{itemize}
\item[i)] There exist three nonzero polynomials, $\phi$, $\psi$ and $\rho$, $(\psi,\rho)$ being an admissible pair, such that 
\begin{align*}
{\bf D}_q (\phi {\bf u}) =\psi {\bf u},\qquad {\bf S}_q (\phi {\bf u}) =\rho {\bf u}. 
\end{align*} 

\item[ii)] There exist $s\in\mathbb{N}$, complex numbers $(a_{n,j})_{j=0}^n$, and a polynomial $\phi$ such that 
\begin{align}\label{charact-split}
\phi\mathcal{D}_qP_n=\sum_{j=n-s} ^{n+\deg \phi -1} a_{n,j}P_j, 
\end{align} 
with $a_{n,n-s}\neq 0$ for all  $n\geq s$.
\end{itemize}
\end{theorem}

\begin{proof}
$i) \implies ii)$: Write  $\rho(x)=a_r\,x^r+ (\text{lower degree}$ $\text{terms})$ and $\psi(x)=b_s\,x^s+ (\text{lower degree terms})$ $(r\in \mathbb{N}\setminus\{0\}, p\in \mathbb{N})$. Set
$d_n=a_r \alpha_{n-1}+b_s \gamma_n$. Clearly, 
\begin{align*}
\phi\mathcal{D}_qP_n=\sum_{j=0} ^{n+\deg \phi -1} a_{n,j}P_j, 
\end{align*} 
where
\begin{align*}
a_{n,j}= \frac{\left\langle {\bf u}, \phi P_j \mathcal{D}_qP_n  \right\rangle}{\left\langle {\bf u},  P_j ^2 \right\rangle}.
\end{align*}
From \eqref{def-fDxg} we get
\begin{align*}
\left\langle {\bf u},  P_j ^2 \right\rangle a_{n,j}&= \left\langle {\bf u}, \phi P_j \mathcal{D}_qP_n  \right\rangle=\left\langle \phi {\bf u},  P_j \mathcal{D}_qP_n  \right\rangle\\[7pt]
&=\left\langle \phi {\bf u}, \mathcal{D}_q\left(\left(\mathcal{S}_q P_j-\alpha ^{-1}\texttt{U}_1 \mathcal{D}_qP_j \right)P_n\right)-\alpha ^{-1}\mathcal{S}_q (P_n\mathcal{D}_q P_j)   \right\rangle\\[7pt]
&=-\left\langle {\bf D}_q(\phi {\bf u}), \left(\mathcal{S}_q P_j-\alpha ^{-1}\texttt{U}_1 \mathcal{D}_qP_j \right)P_n \right\rangle -\alpha ^{-1}\left\langle {\bf S}_q(\phi {\bf u}), P_n\mathcal{D}_qP_j \right\rangle \\[7pt]
&=-\left\langle {\bf u}, \left(\psi\left(\mathcal{S}_q P_j-\alpha ^{-1}\texttt{U}_1 \mathcal{D}_qP_j \right)  +\alpha^{-1}\rho \mathcal{D}_qP_j \right)P_n \right\rangle.
\end{align*}
There is no loss of generality in assuming $r-1\leq s$. For $r-1<s$, we have
\begin{align*}
-\alpha \left\langle {\bf u},  P_j ^2 \right\rangle a_{n,j}=\left\{
    \begin{array}{ll}
        a_r \alpha_{n-s-1} \left\langle {\bf u}, P_n ^2 \right\rangle, & j=n-s, \\[7pt]
        0, &  j<n-s.
    \end{array}
\right.
\end{align*}
and for $r-1=s$, we get
\begin{align*}
-\alpha \left\langle {\bf u},  P_j ^2 \right\rangle a_{n,j}=\left\{
    \begin{array}{ll}
        d_{n-s}\left\langle {\bf u}, P_n ^2 \right\rangle,  &j=n-s, \\[7pt]
        0, &j<n-s.
    \end{array}
\right.
\end{align*}
Hence $a_{n,n-s}\neq 0$, for $n\geq s$ and $ii)$ follows.

$ii) \implies i)$: Let $({\bf a}_n)_{n\geq 0}$ be the dual basis associated to $(P_n)_{n\geq 0}$. Note that $ii)$ yields
\begin{align*}
\left\langle {\bf D}_q(\phi {\bf a}_n),P_j  \right\rangle &=-\left\langle  {\bf a}_n,\phi \mathcal{D}_q P_j  \right\rangle =-\sum_{l=j-s} ^{j+\deg \phi -1} a_{j,l}\left\langle {\bf a}_n, P_l  \right\rangle\\[7pt]
& =\left\{
    \begin{array}{ll}
        -a_{j,n},  & n-\deg \phi +1\leq j\leq n+s, \\[7pt]
        0, & \mbox{otherwise.} 
    \end{array}
\right.
\end{align*}
Writing
\begin{align*}
{\bf D}_q(\phi {\bf a}_n) = \sum_{j=0} ^{\infty} \left\langle {\bf D}_q(\phi {\bf a}_n),P_j  \right\rangle {\bf a}_j,
\end{align*}
in the sense of the weak dual topology in $\mathcal{P}'$, and taking into account that $\left\langle {\bf u},P_n ^2 \right\rangle {\bf a}_n =P_n {\bf u}$, we get
\begin{align*}
{\bf D}_q(\phi P_n {\bf u}) =R_{n+s} {\bf u},\quad R_{n+s}=-\left\langle {\bf u},P_n ^2 \right\rangle \sum_{j=n-\deg \phi +1} ^{n+s}    \frac{a_{j,n}}{\left\langle {\bf u},P_j ^2  \right\rangle}P_j.
\end{align*}
(Note that $R_{n+s}$ is a polynomial of degree $n+s$.) Taking $n=0$ and $n=1$ in the above expression, we have 
\begin{align}
{\bf D}_q(\phi {\bf u}) &=R_{s} {\bf u},\label{eq-Rs}\\[7pt]
{\bf D}_q(\phi P_1 {\bf u}) &=R_{s+1} {\bf u}.\label{eq-Rs+1}
\end{align}
From \eqref{eq-Rs+1}, and using \eqref{def-Dx-fu} and \eqref{eq-Rs}, we obtain
\begin{align*}
\alpha R_{s+1} {\bf u}&=\alpha {\bf D}_q(\phi P_1 {\bf u})= \big(\alpha \mathcal{S}_qP_1 -\texttt{U}_1\mathcal{D}_qP_1\big){\bf D}_q(\phi {\bf u}) +\mathcal{D}_qP_1{\bf S}_q(\phi {\bf u})\\[7pt]
&=(x-\alpha B_0)R_s{\bf u} +{\bf S}_q(\phi {\bf u}).
\end{align*}
Hence
\begin{align}
{\bf S}_q(\phi {\bf u}) = \big(\alpha R_{s+1}-(x-\alpha B_0)R_s \big){\bf u}.\label{eq-Sxfinal}
\end{align}
Note that $\alpha R_{s+1}-(x-\alpha B_0)R_s\neq 0$. To obtain a contradiction, suppose that the last assertion is false. Consequently, $\phi {\bf u}=0$ with $\phi \neq 0$  and ${\bf u}$ regular, which is impossible. We claim that $\big(R_s, \alpha R_{s+1}-(x-\alpha B_0)R_s\big)$ is an admissible pair. According to Definition \ref{adm}, this is equivalent to showing that
$d_n=a_r \alpha_{n-1} +b_s\gamma_n\neq 0$,
where
$$
a_r=-\frac{\left\langle {\bf u}, P_0 ^2\right\rangle }{\left\langle {\bf u}, P_{s} ^2\right\rangle} a_{s,0}, \quad b_s=-\alpha \frac{\left\langle {\bf u}, P_1 ^2\right\rangle }{\left\langle {\bf u}, P_{s+1} ^2\right\rangle} a_{s+1,1} -a_r.
$$
For $\deg\,(\alpha R_{s+1}-(x-\alpha B_0)R_s)<s+1$, we have $b_s=0$, and so $d_n=a_r \alpha_{n-1} \neq 0$. Assume $\deg\,(\alpha R_{s+1}-(x-\alpha B_0)R_s)=s+1$. (Note that in this case $b_s\neq 0$.) We now claim
 \begin{align}\label{eq-with-Sq}
\phi\mathcal{S}_qP_n=\dps \sum_{j=n-s-1} ^{n+\deg \phi} \widetilde{a}_{n,j}P_j, \quad \widetilde{a}_{n,n-s-1}=-\alpha a_{n,n-s}C_{n-s} +a_{n-1,n-s-1}C_n,
\end{align}
where we have assumed that $(P_n)_{n\geq 0}$ satisfies \eqref{TTRR}. Indeed, apply $\phi \mathcal{D}_q$ to \eqref{TTRR} to obtain
\begin{align*}
\phi (x)\mathcal{S}_qP_n(x) &=\phi (x)\left(-\alpha x\mathcal{D}_qP_{n}(x) +\mathcal{D}_qP_{n+1}(x)+B_n\mathcal{D}_qP_{n}(x)+C_n\mathcal{D}_qP_{n-1}(x)\right)\\[7pt]
&=-\alpha x\sum_{j=n-s} ^{n+\deg \phi-1} a_{n,j}P_j(x)+\sum_{j=n-s+1} ^{n+\deg \phi} a_{n+1,j}P_j(x) \\[7pt]
&\quad +B_n\sum_{j=n-s} ^{n+\deg \phi-1} a_{n,j}P_j(x) +C_n \sum_{j=n-s-1} ^{n+\deg \phi-2} a_{n-1,j}P_j(x),
\end{align*}
and \eqref{eq-with-Sq} follows by using \eqref{TTRR}. We also claim that
\begin{align}
&a_{n,n-s}= \left(k_1q^{n/2}+k_2q^{-n/2} \right)\prod_{j=n-s+1} ^n C_j,\qquad n\geq s.\label{expr-ann-s}\\[7pt]
&2 \widetilde{a}_{n,n-s-1}=-(q^{1/2}-q^{-1/2})\left(k_1q^{n/2}-k_2q^{-n/2} \right)\prod_{j=n-s} ^n C_j,\qquad n\geq s+1,\label{expr-hatann-s}
\end{align}
where $k_1$ and $k_2$ are complex numbers. Indeed, apply $\phi \mathcal{S}_q$ to \eqref{TTRR} and use \eqref{def-Sx-fg} to obtain
\begin{align*}
&\texttt{U}_2(x) \phi(x) \mathcal{D}_qP_n(x)+\alpha x \phi(x) \mathcal{S}_qP_n(x)\\[7pt]
&\quad =\phi(x) \mathcal{S}_qP_{n+1}(x) +B_n\phi(x) \mathcal{S}_qP_n(x)+C_n\phi(x) \mathcal{S}_qP_{n-1}(x).
\end{align*}
Combining \eqref{charact-split}, \eqref{eq-with-Sq} and \eqref{TTRR}, we obtain
$$
\sum_{j=n-s-2}^{n+\deg \phi +1} r_{n,j} P_j=0.
$$
Since $(P_n)_{n\geq 0}$ is a free system, we have $r_{n,j}=0$ for all $j$.  By identifying the coefficient of $P_{n-s-2}$, we find 
\begin{align*}
0=r_{n, n-s-2} = (\alpha ^2-1)a_{n,n-s}C_{n-s}C_{n-s-1}+\alpha \widetilde{a}_{n,n-s-1}C_{n-s-1}-\widetilde{a}_{n-1,n-s-2}C_n.
\end{align*} 
Using the expression of $\widetilde{a}_{n,n-s-1}$ given in \eqref{eq-with-Sq}, we get the following second order linear homogeneous equation:
\begin{align}\label{de}
y(n) -2\alpha y(n-1) + y(n-2)=0, 
\end{align}
where
$$
y(n)=\displaystyle \frac{a_{n,n-s}}{\prod_{j=n-s+1} ^n C_j}, \qquad n\geq s.
$$
Note that $q^{1/2}$ and $q^{-1/2}$ are the solutions of the characteristic equation of \eqref{de} and, therefore, we find
\begin{align*}
y(n) =k_1q^{n/2}+k_2q^{-n/2},
\end{align*}
and \eqref{expr-ann-s} follows. Moreover, from the expression of $\widetilde{a}_{n,n-s-1}$ given in \eqref{eq-with-Sq}, \eqref{expr-hatann-s} follows. Finally, using \eqref{expr-ann-s} and \eqref{expr-hatann-s}, we obtain
\begin{align*}
d_n&=-\frac{\left\langle {\bf u}, P_0 ^2\right\rangle }{\left\langle {\bf u}, P_s ^2\right\rangle }a_{s,0} \alpha_{n-1}+\frac{\left\langle {\bf u}, P_0 ^2\right\rangle }{\left\langle {\bf u}, P_{s+1} ^2\right\rangle }\widetilde{a}_{s+1,0} \gamma_n\\[7pt]
&=\frac{\left\langle {\bf u}, P_0 ^2\right\rangle }{\left\langle {\bf u}, P_{s+1} ^2\right\rangle } \left(-a_{s,0}C_{s+1}\alpha_{n-1} +\widetilde{a}_{s+1,0}\gamma_n  \right)\\[7pt]
&=-\frac{1}{2}\left(2\alpha_{n-1}(k_1q^{s/2}+k_2q^{-s/2}) +(q^{n/2}-q^{-n/2})(k_1q^{(s+1)/2}-k_2q^{-(s+1)/2})   \right)\\[7pt]
&=-\alpha \big(k_1q^{(n+s)/2}+k_2q^{-(n+s)/2}  \big) \\[7pt]
&=-\alpha\,a_{n+s,n} \prod_{j=1} ^{n+s} C_j^{-1}\neq 0,
\end{align*}
and so $(R_s,\alpha R_{s+1}-(x-\alpha B_0)R_s)$ is an admissible pair. Thus, $i)$ follows from \eqref{eq-Rs} and \eqref{eq-Sxfinal}, and the theorem is proved.
\end{proof}

\begin{obs}
A regular linear form ${\bf u}$ satisfying Theorem \ref{main-theorem} $i)$ is classical or semiclassical. Indeed, using \eqref{DxnSx-u} and \eqref{def-fD_x-u} we get
\begin{align*}
{\bf D}_q(\rho {\bf u})&= {\bf D}_q{\bf S}_q(\phi {\bf u})=(\alpha -\alpha ^{-1}){\bf S}_q{\bf D}_q(\phi {\bf u}) +\alpha ^{-1}\texttt{U}_1{\bf D}_q ^2 (\phi {\bf u})\\[7pt]
&=(\alpha -\alpha ^{-1}){\bf S}_q(\psi {\bf u}) +\alpha ^{-1}\texttt{U}_1{\bf D}_q (\psi {\bf u})\\[7pt]
&=(\alpha -\alpha ^{-1}){\bf S}_q(\psi {\bf u}) +\alpha ^{-1}(\alpha {\bf D}_q(\texttt{U}_1\psi {\bf u}) -(\alpha ^2-1){\bf S}_q(\psi {\bf u}))\\[7pt]
&=\alpha {\bf S}_q(\psi {\bf u}) +{\bf D}_q(\texttt{U}_1\psi {\bf u}).
\end{align*}
Thus  ${\bf D}_q\big( (\rho -\texttt{U}_1\psi){\bf u} \big)=\alpha {\bf S}_q(\psi {\bf u})$ as claimed. Theorem \ref{main-theorem} is the analogue of \cite[Theorem 3.1]{M87} $($see also \cite{M91}$)$, from which a distributional version of \cite[Theorem 1.1]{BLN87} follows. 
\end{obs}


Although Theorem \ref{main-theorem} could be very useful in many situations it is little precise regarding the classical or semiclassical character of the linear form.  In the following theorem we will be more precise in this sense, but in counterpart we lose the direct connection with the equation \eqref{NUL-Pearson}.
\begin{theorem}\label{theo-2b}
Let ${\bf u} \in \mathcal{P}'$ be regular and let $(P_n)_{n\geq 0}$  denote the corresponding sequence of orthogonal polynomials satisfying \eqref{TTRR}. Suppose that there exist a nonnegative integer $s$, complex numbers $(a_{n,j})_{j=0}^n$, and a polynomial $\phi$ such that
\begin{align}\label{2-f}
\phi\mathcal{D}_{q}P_{n}=\sum_{j=n-s} ^{n+\deg \phi -1}  a_{n,j}P_{j},
\end{align}
with $a_{n.n-s}\neq 0$ for all $n\geq s$.
Then there exist $\Phi \in \mathcal{P}_{s+1}\setminus\mathcal{P}_{-1}$ and $\Psi \in \mathcal{P}_s\setminus\mathcal{P}_{-1}$ such that 
\begin{align}\label{eq-class-of-u-AW-case}
\Phi\, {\bf D}_q{\bf u} =\Psi\, {\bf S}_q{\bf u},
\end{align}
where $\Phi$ and $\Psi$ never have more than $s-1$ common zeros. If $\Phi$ and $\Psi$ have $s-1$ common zeros, then ${\bf u}$ is ${\bf D}_q$-classical. Otherwise, ${\bf u}$ is ${\bf D}_q$-semiclassical of class $s-1-r$, $r$ being the number of common zeros of $\Phi$ and $\Psi$.
\end{theorem}
\begin{proof}
We can now proceed analogously to the proof of Theorem \ref{main-theorem} to obtain
\begin{align*}
{\bf D}_q(\phi P_n {\bf u}) =-Q_{n+s} {\bf u},\quad Q_{n+s}=\left\langle {\bf u},P_n ^2 \right\rangle \sum_{j=n-\deg \phi +1} ^{n+s}    \frac{a_{j,n}}{\left\langle {\bf u},P_j ^2  \right\rangle}P_j.
\end{align*}
(Note that $Q_{n+s}$ is a polynomial of degree $n+s$.) Taking $n=0$ and $n=1$, in the above expression, we have 
\begin{align}
{\bf D}_q(\phi {\bf u}) &=-Q_{s} {\bf u},\label{eq-Rs-f}\\[7pt]
{\bf D}_q(\phi P_1 {\bf u}) &=-Q_{s+1} {\bf u}.\label{eq-Rs+1-f}
\end{align}
Using \eqref{eq-Rs-f} and \eqref{def-Dx-fu}, we have
\begin{align*}
-\alpha Q_{s+1}(x){\bf u}&=\alpha {\bf D}_q\big(P_1(x)\phi{\bf u})=(x-\alpha B_0){\bf D}_q(\phi {\bf u})+{\bf S}_q(\phi {\bf u})\\[7pt]
&=-(x-\alpha B_0)Q_s(x){\bf u} +{\bf S}_q (\phi {\bf u}),
\end{align*}
and so $$( (x-\alpha B_0)Q_s(x)-\alpha Q_{s+1}(x)) {\bf u}={\bf S}_q (\phi {\bf u}).$$ Applying ${\bf D}_q$ to the above equation, and using \eqref{DxnSx-u}, \eqref{eq-Rs-f}, and \eqref{def-fD_x-u}, we can assert that
\begin{align*}
&-\alpha{\bf D}_q(( (x-\alpha B_0)Q_s(x)-\alpha Q_{s+1}(x)) {\bf u})\\[7pt]
&\quad =-\alpha {\bf D}_q{\bf S}_q\big(\phi {\bf u} \big)=-(2\alpha ^2-1){\bf S}_q{\bf D}_q\big(\phi {\bf u} \big) -\texttt{U}_1{\bf D}_q^2\big(\phi {\bf u} \big)\\[7pt]
&\quad =(2\alpha ^2-1){\bf S}_q\big(Q_s(x){\bf u} \big) +\texttt{U}_1{\bf D}_q\big(Q_s(x){\bf u} \big)=\alpha ^2{\bf S}_q\big(Q_s(x){\bf u} \big) +\alpha {\bf D}_q\big(\texttt{U}_1Q_s(x){\bf u} \big).
\end{align*}
Therefore,
\begin{align}\label{320-f}
{\bf D}_q\big((Q_{s+1}(x)-(\alpha x-B_0)Q_s(x)){\bf u}\big)={\bf S}_q (Q_s(x){\bf u}),
\end{align}
and $\mathbf{u}$ is classical or semiclassical of class at most $s-1$. Let us now rewrite \eqref{320-f} in the form \eqref{eq-class-of-u-AW-case} to distinguish between cases. Using \eqref{def-Dx-fu} and \eqref{def-Sx-fu}, \eqref{320-f} becomes
\begin{align*}
&(\alpha \mathcal{S}_qR_{s+1} -\texttt{U}_1\mathcal{D}_qR_{s+1}+\big(\texttt{U}_1 ^2-\alpha ^2\texttt{U}_2\big)\mathcal{D}_qQ_s){\bf D}_q{\bf u}\\[7pt]
&\quad = \big(\texttt{U}_1\mathcal{D}_qQ_s+\alpha \mathcal{S}_qQ_s-\mathcal{D}_qR_{s+1} \big){\bf S}_q{\bf u},
\end{align*}
where $R_{s+1}(x)=Q_{s+1}(x)-(\alpha x-B_0)Q_s(x)$. Consequently, \eqref{eq-class-of-u-AW-case} follows with
\begin{align}
\label{Phi}\Phi&=\alpha \mathcal{S}_qR_{s+1} -\texttt{U}_1\mathcal{D}_qR_{s+1}+(\texttt{U}_1 ^2 -\alpha ^2\texttt{U}_2 )\mathcal{D}_qQ_s,\\[7pt]
\label{Psi}\Psi&=\texttt{U}_1\mathcal{D}_qQ_s+\alpha \mathcal{S}_qQ_s-\mathcal{D}_qR_{s+1}.
\end{align}
Of course, $\Phi(x)\not =0$ and $\Psi(x)\not=0$, otherwise ${\bf u}={\bf 0}$, which contradicts the regularity of $\mathbf{u}$.  Without restriction of generality, let us assume $\Phi(x)=(x -1)\Psi(x)$. From \eqref{eq-class-of-u-AW-case}, and using \eqref{def-fD_x-u} and \eqref{def-fS_x-u}, we get
\begin{align*}
{\bf D}_q\big(1/2(\alpha\,x-1){\bf u} \big)={\bf S}_q{\bf u}.
\end{align*}
By Theorem \ref{regAW}, this leads to a contradiction with the regularity of $\mathbf{u}$ ---$a=d=0$ in the notation of Theorem \ref{regAW} and so $d_n=0$ therein---, and the first part of the theorem follows. Now suppose that $\Phi=\rho_r \phi$ and $ \Psi=\rho_r\psi$, $r< s$ where $\rho_r \in \mathcal{P}_r$, $\phi \in \mathcal{P}_{s-r+1}$ and $\psi\in \mathcal{P}_{s-r}$. Hence \eqref{eq-class-of-u-AW-case} reduces to $\phi \, {\bf D}_q{\bf u}=\psi \, {\bf S}_q {\bf u}$ and, therefore, using \eqref{def-fD_x-u} and \eqref{def-fS_x-u}, we have
$$
{\bf D}_q\left((\mathcal{S}_q\phi +\texttt{U}_2 \mathcal{D}_q\psi)\,{\bf u} \right)= {\bf S}_q\left((\mathcal{S}_q\psi+\mathcal{D}_q\phi)\,{\bf u}\right).
$$
Since  $\mathcal{S}_q\phi$ has degree at most $s-r+1$, and $\mathcal{S}_q\psi$ and $\mathcal{D}_q\phi$ have degree at most $s-r$, $\mathbf{u}$ is classical whenever $r=s-1$ or semiclassical of class at most $s-r-1$ whenever $r<s-1$. Assume that \eqref{eq-class-of-u-AW-case} holds with $\Phi$ and $\Psi$ being coprime, i.e., $r=0$. To obtain a contradiction, suppose that ${\bf u}$ is semiclassical of class at most $s-2$: there exists $\phi\in \mathcal{P}_{s}$ and $\psi \in \mathcal{P}_{s-1}$ such that \eqref{NUL-Pearson} holds. Taking into account \eqref{def-Dx-fu} and \eqref{def-Sx-fu}, \eqref{NUL-Pearson} holds if and only if 
\begin{align}\label{aux-f}
\widetilde{\Phi} \,{\bf D}_q{\bf u}=\widetilde{\Psi}\,{\bf S}_q {\bf u},
\end{align}
where $\widetilde{\Phi} =\alpha \mathcal{S}_q\phi -\texttt{U}_1\mathcal{D}_q\phi+(\texttt{U}_1 ^2 -\alpha^2 \texttt{U}_2)\mathcal{D}_q\psi$ and $\widetilde{\Psi}=\alpha \mathcal{S}_q\psi+\texttt{U}_1\mathcal{D}_q\psi-\mathcal{D}_q\phi$. Combining \eqref{eq-class-of-u-AW-case} with \eqref{aux-f} yields
 $$
 (\widetilde{\Psi}\, \Phi-\widetilde{\Phi}\, \Psi)\, {\bf D}_q{\bf u}=\mathbf{0}.
 $$
By the regularity of $\mathbf{u}$, and the fact that  $\Phi$ and $\Psi$ are coprime, $\Phi=a \widetilde{\Phi} $ and $\Psi=a \widetilde{\Psi}$ $(a\in \mathbb{C}\setminus\{0\})$, and so
 $$
Q_s=a\, \psi,
 $$
 which is impossible. Thus ${\bf u}$ is semiclassical of class $s-1$. The same conclusion can be drawn for $r\not=0$ and the theorem is proved.

\end{proof}

\section{Counterexamples}\label{Sec4}

As an application of a particular case of Theorem \ref{theo-2b}, we disprove Conjecture \ref{conj1} and Conjecture \ref{conj2}.

\begin{proposition}\label{Askey1}
Let $(P_n)_{n\geq 0}$  be a sequence of orthogonal polynomials satisfying \eqref{TTRR} with
\begin{align*}
B_n=0, \qquad C_{n}= \frac{1}{4}\big(1-(-1)^{n}q^{n/2}  \big)\big(1-(-1)^{n}q^{(n-1)/2}  \big).
\end{align*}
Then $(P_n)_{n\geq 0}$ is $\mathcal{D}_{q}$-semiclassical of class two and the corresponding linear form satisfies \eqref{NUL-Pearson} with
\begin{align*}
\phi(x)&=-\frac{1}{8}(1-q^{-1})^2\big(4x^4-(q+5)x^2+q+1),\\[7pt]
\psi(x)&=\frac{1}{4}(q-1)q^{-3/2}x(4x^2-3-q).
\end{align*}
\end{proposition}
\begin{proof}
We claim that $(P_n)_{n\geq 0}$ satisfies
\begin{align}
\label{true-eq-1-prime} \mathcal{S}_q P_n&=\alpha_n P_n+b_nC_{n-1} P_{n-2}, \\[7pt]
\texttt{U}_2 \mathcal{D}_qP_n&=a_n P_{n+1}+c_n P_{n-1}+d_n P_{n-3},\label{true-eq-2}
\end{align}
with 
\begin{align*}
a_n&= (\alpha ^2 -1)\gamma_n,\\[7pt]
b_n&=-\frac{1}{2}\big(1-(-1)^nq^{n/2}  \big)\big((-1)^n-q^{-(n-1)/2} \big),\\[7pt]
c_n&=b_{n+1}C_n -\alpha b_nC_{n-1}-(\alpha ^2 -1)\gamma_nC_n,\qquad d_n=(b_{n-1}C_n -\alpha b_nC_{n-1})C_{n-2}.
\end{align*}
Indeed, we prove this by induction on $n$. For $n=1$, RHS of \eqref{true-eq-1-prime} gives 
$\alpha_1P_1(x) +b_1C_0P_{-1}(x)=\alpha x,$ while LHS gives $\mathcal{S}_qP_1(x)=\mathcal{S}_qx=\alpha x$.   Similarly, for $n=1$, LHS of \eqref{true-eq-2} gives 
$\texttt{U}_2 \mathcal{D}_qP_1=\texttt{U}_2$, while RHS gives 
\begin{align*}
a_1 P_2(x)+c_1P_0(x)+d_1P_{-1}(x)&=a_1P_2(x)+c_1\\[7pt]
&=(\alpha^2-1) (x^2 -C_1) +b_2C_1 -\alpha b_1C_0 -(\alpha^2-1)C_1\\[7pt]
&=(\alpha ^2 -1)(x^2-1).
\end{align*}
 Assuming \eqref{true-eq-1-prime} and \eqref{true-eq-2} hold, with $k$ instead of $n$, for $k=1, 2, \dots, n$, we will prove it for $k=n+1$. Apply $\mathcal{S}_q$ to \eqref{TTRR}, and use \eqref{def-Sx-fg},  to obtain
 \begin{align*}
\mathcal{S}_q(P_{n+1}(x) +C_n P_{n-1}(x))&= \mathcal{S}_q(xP_n(x) )=\texttt{U}_2(x)\mathcal{D}_qP_n(x) +\alpha x\mathcal{S}_qP_n(x).
\end{align*}
From \eqref{true-eq-2} for $n$ and \eqref{true-eq-1-prime} for $n-1$ and $n$, we get
\begin{align}\label{intermediate-1-q}
\mathcal{S}_q P_{n+1}(x)&=a_nP_{n+1}(x)+c_nP_{n-1}(x)+d_nP_{n-3}(x) \\[7pt]
&\quad +\alpha x (\alpha_n P_{n}(x)+b_nC_{n-1}P_{n-2}(x)) \nonumber\\[7pt]
&\quad -C_n(\alpha_{n-1} P_{n-1}(x)+b_{n-1}C_{n-2}P_{n-3}(x)).\nonumber
\end{align}
Now using \eqref{TTRR},  \eqref{intermediate-1-q} becomes
\begin{align*}
\mathcal{S}_qP_{n+1}&=(a_n +\alpha \alpha_n)P_{n+1}\\[7pt]
& +(c_n +\alpha \alpha_nC_n+\alpha b_nC_{n-1}-\alpha_{n-1}C_n)P_{n-1}\\[7pt]
& +(d_n+\alpha b_nC_{n-1}C_{n-2}-b_{n-1}C_nC_{n-2})P_{n-3}.
\end{align*}
The reader should convince himself that the following relations hold:
\begin{align*}
\alpha_{n+1}&=a_n +\alpha \alpha_n,\\[7pt]
b_{n+1}C_{n+1}&=c_n +\alpha \alpha_nC_n+\alpha b_nC_{n-1}-\alpha_{n-1}C_n,  \\[7pt]
0&=d_n+\alpha b_nC_{n-1}C_{n-2}-b_{n-1}C_nC_{n-2}.
\end{align*}
This gives $\mathcal{S}_qP_{n+1}=\alpha_{n+1}P_{n+1}+b_{n+1}C_{n}P_{n-1}$, 
and \eqref{true-eq-1-prime} holds for $n+1$. Similarly, apply $\texttt{U}_2 \mathcal{D}_q$ to \eqref{TTRR}, and use \eqref{def-Dx-fg},  to obtain
\begin{align*}
\texttt{U}_2(x) \mathcal{D}_q (P_{n+1}(x)+C_n P_{n-1}(x))=\texttt{U}_2(x) \mathcal{D}_q (x P_n(x))=\texttt{U}_2(x) (\mathcal{S}_q P_n(x)+\alpha x\mathcal{D}_q P_n(x))
\end{align*}
or, using \eqref{true-eq-1-prime} for $n$, 
\begin{align}
\texttt{U}_2(x) \mathcal{D}_qP_{n+1}(x)&=\texttt{U}_2(x) \mathcal{S}_q P_n(x)+\alpha x\texttt{U}_2(x) \mathcal{D}_q P_n(x)-C_n\texttt{U}_2(x) \mathcal{D}_q P_{n-1}(x)\label{Ip1}\\[7pt]
&=\texttt{U}_2(x)\big(\alpha_nP_n(x)+b_nC_{n-1}P_{n-2}(x)  \big)+\alpha x\texttt{U}_2(x) \mathcal{D}_q P_n(x)\nonumber\\[7pt]
&-C_n\texttt{U}_2(x) \mathcal{D}_q P_{n-1}(x).\nonumber
\end{align}
From \eqref{TTRR} it follows that
\begin{align*}
\texttt{U}_2P_n=(\alpha ^2-1)(P_{n+2}+(C_{n+1}+C_n-1)P_{n}+C_nC_{n-1}P_{n-2}).
\end{align*}
Combining the above equation with \eqref{true-eq-2} for $n-1$ and $n$, \eqref{Ip1} becomes
\begin{align*}
\texttt{U}_2 \mathcal{D}_q P_{n+1}&=\big((\alpha ^2-1)\alpha_n +\alpha a_n\big)P_{n+2}\\[7pt]
&\quad+\big( (\alpha ^2-1)\alpha_n (C_n+C_{n+1}-1) +(\alpha ^2-1)b_nC_{n-1}+\alpha a_nC_{n+1} \\[7pt]
&\quad\quad+\alpha c_n -a_{n-1}C_n \big)P_{n}\\[7pt]
&\quad+\big((\alpha ^2-1)\alpha_nC_nC_{n-1}+(\alpha ^2-1)b_nC_{n-1}(C_{n-1}+C_{n-2}-1)\\[7pt]
&\quad\quad +\alpha c_nC_{n-1}+\alpha d_n-c_{n-1}C_n  \big)P_{n-2}\\[7pt]
&\quad+\big(\alpha d_nC_{n-3}-d_{n-1}C_n+(\alpha ^2-1)b_nC_{n-1}C_{n-2}C_{n-3} \big)P_{n-4}.
\end{align*}
The reader again should convince himself that the following relations hold:
\begin{align*}
a_{n+1}&=(\alpha ^2-1)\alpha_n +\alpha a_n,\\[7pt]
c_{n+1}&=(\alpha ^2-1)\alpha_n (C_n+C_{n+1}-1) +(\alpha ^2-1)b_nC_{n-1}+\alpha a_nC_{n+1} \\[7pt]
&\quad+\alpha c_n -a_{n-1}C_n,\\[7pt]
d_{n+1}&=(\alpha ^2-1)\alpha_nC_nC_{n-1}+(\alpha ^2-1)b_nC_{n-1}(C_{n-1}+C_{n-2}-1)\\[7pt]
&\quad +\alpha c_nC_{n-1}+\alpha d_n-c_{n-1}C_n, \\[7pt]
0&=\alpha d_nC_{n-3}-d_{n-1}C_n+(\alpha ^2-1)b_nC_{n-1}C_{n-2}C_{n-3}.
\end{align*}
We thus get $$\texttt{U}_2\mathcal{D}_q P_{n+1}=a_{n+1}P_{n+2}+c_{n+1}P_{n}+d_{n+1}P_{n-2},$$ as claimed. 
Observe from \eqref{true-eq-2} that $(P_n)_{n\geq 0}$ satisfies the hypotheses of Theorem \ref{theo-2b} with $B_n=0$ and $\phi=\texttt{U}_2$. Note also that
\begin{align*}
C_1&=\frac12(1+q^{1/2}),\qquad &C_2&=\frac14(1-q)(1-q^{1/2}),\qquad \\
 C_3&=\frac14(1+q)(1+q^{3/2}),\qquad &C_4&=\frac14(1-q^2)(1-q^{3/2}).
\end{align*}
 Under the notation of Theorem \ref{theo-2b} and its proof, we get
\begin{align}
\label{Q3} Q_3(x) &= \frac{c_1}{C_1}P_1(x) +\frac{d_3}{C_3C_2C_1}P_3(x)=\frac{1}{4}(q-1)q^{-3/2}x(4x^2-3-q),\\[7pt]
\nonumber Q_4(x) &=\frac{c_2}{C_2} P_2(x) + \frac{d_4}{C_4C_3C_2}P_4(x)= \frac{1}{8}(q-1)q^{-2}(8x^4-8x^2+1-q^2),\\[7pt]
\label{R4}R_4(x)&=Q_4(x) -\alpha x Q_3(x)=-\frac{1}{8}(1-q^{-1})^2(4x^4-(q+5)x^2+q+1).
\end{align}
Taking into account that $\mathcal{S}_qx=\alpha x$, $\mathcal{D}_q x^2 =2\alpha x$,
\begin{align*}
\mathcal{S}_q x^2&= (2\alpha ^2 -1)x^2 +1-\alpha ^2,\\[7pt]
\mathcal{D}_q x^3&= (4\alpha ^2 -1)x^2+1-\alpha ^2, \quad \mathcal{S}_q x^3 = \alpha (4\alpha ^2-3)x^3 +3\alpha (1-\alpha ^2)x,\\[7pt]
\mathcal{D}_q x^4&= 4\alpha(2\alpha^2-1)x^3+4\alpha(1-\alpha^2)x,\\[7pt]
\mathcal{S}_q x^4&=(8\alpha^4-8\alpha^2+1)x^4+2(1-\alpha^2)(4\alpha^2-1)x^2+(1-\alpha^2)^2  ,
\end{align*}
from \eqref{Phi} and \eqref{Psi}, we have
\begin{align*}
\Phi(x)&=-\frac{1}{16}(q-1)^2q^{-3/2}(8 q x^4 -2(q^2+4q+1)x^2 +(q+1)^2),\\[7pt]
\Psi(x)&=\frac{1}{4}(q^{1/2}-q^{-1/2})(4 q x^2-3q-1)x.
\end{align*}
Finally, since $\Phi$ and $\Psi$ are coprime, by Theorem \ref{theo-2b}, \eqref{320-f}, \eqref{Q3} and \eqref{R4},  the result follows.
\end{proof}

\begin{obs}
The  structure relation \eqref{true-eq-2} is of type \eqref{0.2Dq-general} with $\phi=\texttt{U}_2$, $M=3$, and $N=1$. Consequently, the semiclassical orthogonal polynomials given in Proposition \ref{Askey1} disprove Conjecture \ref{conj1}.
\end{obs}

\begin{coro}\label{corollary}
Assume the hypotheses and notation of Proposition \ref{Askey1}. Then $(P_n)_{n\geq 0}$ satisfies
\begin{align}
\label{eq-main-eq} (\alpha ^2-1)^2(x^2 -\alpha ^2)(1-x^2) \mathcal{D}_q^2P_n(x) &=-(\alpha ^2-1)^2\gamma_n\gamma_{n-1} P_{n+2}(x) +d_{n,1}P_{n}(x)\\[7pt]
\nonumber &+d_{n,2}P_{n-2}(x) +d_{n,3}P_{n-4}(x)+d_{n,4}P_{n-6}(x),
\end{align}
with
\begin{align*}
d_{n,1}&=a_nc_{n+1}+a_{n-1}c_n-2\alpha(\alpha ^2-1)(\alpha_n ^2-1)(C_{n+1}+C_n-1)\\[7pt]
&\quad -4\alpha^2(\alpha^2-1)\alpha_{n-1}b_nC_{n-1}  ,\\[7pt]
d_{n,2}&=a_nd_{n+1}+c_nc_{n-1}+a_{n-3}d_n-2\alpha(\alpha ^2-1)(\alpha_n ^2-1)C_nC_{n-1}\\[7pt]
&\quad-4\alpha ^2(\alpha ^2-1)\alpha_{n-1}b_nC_{n-1}(C_{n-1}+C_{n-2}-1)-2\alpha (\alpha ^2-1)b_nb_{n-2}C_{n-1}C_{n-3} ,\\[7pt]
d_{n,3}&=c_nd_{n-1}+c_{n-3}d_n-4\alpha ^2(\alpha^2-1)\alpha_{n-1}b_nC_{n-1}C_{n-2}C_{n-3}\\[7pt]
&\quad -2\alpha(\alpha ^2-1)b_nb_{n-2}C_{n-1}C_{n-3}(C_{n-3}+C_{n-4}-1)  ,\\[7pt]
d_{n,4}&= -4\alpha ^2q^{-(n-3)/2}C_nC_{n-1}C_{n-2}C_{n-3}C_{n-4}C_{n-5}  .
\end{align*}
\end{coro}
\begin{proof}
From the previous result, we apply $\texttt{U}_2\mathcal{D}_q$ to \eqref{true-eq-2}, and use \eqref{def-Dx-fg}, to get
\begin{align*}
&\texttt{U}_2\mathcal{S}_q\texttt{U}_2 \mathcal{D}_q ^2P_n +\texttt{U}_2\mathcal{D}_q\texttt{U}_2\mathcal{S}_q\mathcal{D}_qP_n=\texttt{U}_2 \mathcal{D}_q(a_n P_{n+1}+c_n P_{n-1}+d_n P_{n-3}),
\end{align*} 
and since $\mathcal{S}_q\texttt{U}_2=\alpha^2\texttt{U}_2+\texttt{U}_1 ^2$ and $\mathcal{D}_q\texttt{U}_2=2\alpha\texttt{U}_1$, we use again \eqref{true-eq-2} to obtain
\begin{align}\label{star-one}
(\alpha ^2 \texttt{U}_2 &+\texttt{U}_1 ^2)\texttt{U}_2 \mathcal{D}_q ^2P_n +2\alpha \texttt{U}_1\texttt{U}_2\mathcal{S}_q\mathcal{D}_qP_n=a_na_{n+1}P_{n+2}\\[7pt]
&+(a_nc_{n+1}+a_{n-1}c_n)P_n+(a_nd_{n+1}+c_nc_{n-1}+a_{n-3}d_n)P_{n-2}\nonumber\\[7pt]
&+(c_nd_{n-1}+d_nc_{n-3})P_{n-4}+d_nd_{n-3}P_{n-6}.\nonumber
\end{align} 
On the other hand, it is known from \cite[Lemma 2.1]{CMP22a} that 
\begin{align*}
\alpha \mathcal{S}_q ^2P_n&=\mathcal{S}_q(\texttt{U}_1\mathcal{D}_q P_n)+\texttt{U}_2 \mathcal{D}_q ^2P_n +\alpha P_n=\alpha ^2 \textbf{U}_2\mathcal{D}_q ^2P_n+\alpha \texttt{U}_1\mathcal{S}_q\mathcal{D}_qP_n+\alpha P_n,
\end{align*}
where the second equality holds thanks to \eqref{def-Sx-fg}. Now we apply $\mathcal{S}_q$ to \eqref{true-eq-1-prime} using again the same equation and the above equation in order to obtain
\begin{align}\label{star-two}
\alpha \texttt{U}_2(x)\mathcal{D}_q ^2P_n(x)+\texttt{U}_1(x)\mathcal{S}_q\mathcal{D}_qP_n(x)&=(\alpha_n ^2-1)P_n(x)+2\alpha \alpha_{n-1}b_{n}C_{n-1} P_{n-2}(x)\\[7pt]
&+b_nb_{n-2}C_{n-1}C_{n-3}P_{n-4}(x).\nonumber
\end{align} 
Therefore, \eqref{eq-main-eq} holds by combining \eqref{star-one} with \eqref{star-two} in order to eliminate $\mathcal{S}_q\mathcal{D}_qP_n$ and by using \eqref{TTRR}, with
\begin{align*}
d_{n,4}=d_nd_{n-3}-2\alpha (\alpha ^2-1)b_nb_{n-2}C_{n-1}C_{n-3}C_{n-4}C_{n-5}.
\end{align*}
In addition, $b_n=2C_nq^{-(n-1)/2}$ and $d_n=(q-1)q^{-n/2}C_nC_{n-1}C_{n-2}$. Therefore we obtain
\begin{align*}
d_{n,4}=-4\alpha ^2q^{-(n-3)/2}\prod_{j=0} ^5 C_{n-j} \neq 0 .
\end{align*}
The result is then proved.
\end{proof}

\begin{obs}
The  structure relation \eqref{eq-main-eq} is of type \eqref{0.2Dq-general} with $\phi(x)=(\alpha ^2-1)^2(x^2-\alpha^2)(1-x^2)$, $M=6$, and $N=2$. Consequently, the semiclassical orthogonal polynomials given in Proposition \ref{Askey1} disprove Conjecture \ref{conj2}.
\end{obs}

\section{Further results: $D_{q, \omega}$-semiclassical orthogonal polynomials}\label{Sec6}
Although this paper was originally intended to deal with the Askey-Wilson operator, the ideas developed above allow working with other operators. The results of this section are related to the structure relation that appears in \cite[Conjecture 24.7.7]{I05}, and warn the reader of the existence of semiclassical OP in such a problem. Recall that given complex numbers $q$ and $\omega$,
Hahn's operator $D_{q,\omega}:\mathcal{P}\to\mathcal{P}$ is defined by
\begin{align*}
D_{q,\omega}f(x):=\frac{f(qx+\omega)-f(x)}{(q-1)x+\omega},
\end{align*}
where we have fixed $q$ and $\omega$ such that
\begin{align}\label{wq}
|1-q|+|\omega|\not=0, \qquad q\not\in \{0\}\cup\left\{e^{2ij\pi/n}\left|\right.1\leq j\leq n-1, n\in \mathbb{N}\setminus\{0,1\}\right\}.
\end{align}
For every ${\bf u} \in \mathcal{P^*}$ and $f\in \mathcal{P}$, $D_{q,\omega}$ induces
${\bf D}_{q,\omega}:\mathcal{P}^*\to\mathcal{P}^*$ defined by
\begin{align*}
\langle{\bf D}_{q,\omega}{\bf u},f\rangle=
-q^{-1}\langle{\bf u},D_{q,\omega}^*f\rangle,
\end{align*}
where $D_{q,\omega}^*=D_{1/q,-\omega/q}$.

\begin{definition}\label{NUL-def2}  \cite[p. 487]{ACMP2}
${\bf u}\in\mathcal{P}^*$ is called ${\bf D}_{q,\omega}$-classical if it is regular and there exist 
$\phi\in\mathcal{P}_2\setminus \mathcal{P}_{-1}$ and $\psi\in\mathcal{P}_1\setminus \mathcal{P}_{-1}$
such that
\begin{align}\label{NUL-Pearson2}
{\bf D}_{q,\omega} \left( \phi\, \mathbf{u}\right) =\psi\, \mathbf{u}.
\end{align}
$($We will call it simply classical when no confusion can arise.$)$
\end{definition}

\begin{definition}\label{semi2}\cite[p. 855]{ACMP2}
We call a linear form, in $\mathcal{P}^*$, ${\bf D}_{q, \omega} $-semiclassical if it is regular, not ${\bf D}_{q, \omega} $-classical, and there exist
two polynomials $\phi$ and $\psi$ with at least one of them nonzero, such that \eqref{NUL-Pearson2} holds. $($We will call it simply semiclassical when no confusion can arise.$)$
\end{definition}

OP with respect to a ($\mathbf{D}_{q, \omega}$-)semiclassical form of class
$s$ is called ($D_{q, \omega}$-) semiclassical of class $s$. Under the conditions of Definition \ref{semi2}, we define the class of ${\bf u}$ as in Section \ref{Sec2}. The next theorem is the analogue of Theorem \ref{regAW}.
Here we use the standard notation
\begin{align*}
[n]_q=\frac{q^n -1}{q-1}.
\end{align*}

\begin{theorem}\label{Dqw-main-Thm}\cite[Theorem 1.2]{ACMP1}
Suppose that ${\bf u} \in \mathcal{P}^*$ satisfies \eqref{NUL-Pearson2} with $\phi(x)=ax^2+bx+c$ and $\psi(x)=dx+e$.  Then ${\bf u}$ is regular if and only if 
\begin{align*}
d_n\neq0,\qquad \phi\left(-\frac{e_n}{d_{2n}}\right)\neq0,
\end{align*}
for all $n\in \mathbb{N}$, where $d_n=d\,q^n+a[n]_{q}$ and $e_n=eq^n+(\omega d_n+b)[n]_q$. Moreover, $(P_n)_{n\geq 0}$ satisfies \eqref{TTRR} with
\begin{align*}
& B_n=\omega[n]_q+\frac{[n]_qe_{n-1}}{d_{2n-2}}-\frac{[n+1]_qe_{n}}{d_{2n}},\\[7pt]
& C_{n+1}=-\frac{q^{n}[n+1]_qd_{n-1}}{d_{2n-1}d_{2n+1}}\, \phi\left(-\frac{e_n}{d_{2n}}\right). 
\end{align*}
\end{theorem}

In this context we have also an analogue to Theorem \ref{theo-2b} for semiclassical orthogonal polynomials of class one.

\begin{theorem}\label{theo-22b}
Let ${\bf u} \in \mathcal{P}'$ be regular and let $(P_n)_{n\geq 0}$  denote the corresponding sequence of orthogonal polynomials satisfying \eqref{TTRR}. Suppose that there exist complex numbers $c$, $(a_n)_{n\geq 0}$, $(b_n)_{n\geq 0}$ $($$b_n\not=0$$)$, and $(c_n)_{n\geq 0}$ such that
\begin{align}\label{22}
(x-c) D_{q, \omega}P_{n}(x)= a_nP_n(x) +(b_nx+c_n)P_{n-1}(x).
\end{align}
Then
\begin{align*}
{\bf D}_{1/q,-\omega/q}\big((x-c){\bf u}\big)=-\frac{qb_2}{C_2}(x-\lambda_+)(x-\lambda_{-}){\bf u},
\end{align*}
where 
\begin{align*}
\lambda_{\pm} =\frac12 (B_0+B_1)+\frac{c-B_0}{2b_2C_1}C_2 \pm \frac{ 1}{2} \left((B_0-B_1 -\frac{c-B_0}{b_2C_1}C_2)^2 +4C_1\right)^{1/2}.
\end{align*}
 If $q(\omega +qc-\lambda_+)(\omega +qc-\lambda_{-})=-C_2/b_2$, then ${\bf u}$ is ${\bf D}_{1/q,-\omega/q}$-classical. More precisely, $(P_n)_{n\geq 0}$ are the Al-Salam-Carlitz polynomials. Otherwise, ${\bf u}$ is a ${\bf D}_{1/q,-\omega/q}$-semiclassical of class one. Moreover, if 
\begin{align}\label{condt}
b_2C_1^2 =(B_0-c)(b_2(B_1-c)C_1-(B_0-c)C_2),
\end{align} 
then
\begin{align*}
{\bf u}=(x-c)^{-1}{\bf v}+\delta_c,
\end{align*}
${\bf v}$ being the linear form corresponding to the Al-Salam-Carlitz polynomials.
\end{theorem}
\begin{proof}
As in the proof of Theorem \ref{theo-2b}, from \eqref{TTRR} and \eqref{22} we get
\begin{align*}
{\bf D}_{1/q,-\omega/q} \big((x-c){\bf a}_n \big) =-q(a_n+b_n){\bf a}_n-q(c_{n+1}+b_{n+1}B_n){\bf a}_{n+1}-q\,b_{n+2}C_{n+1}{\bf a}_{n+2},
\end{align*}
in the sense of the weak dual topology in $\mathcal{P}'$, $({\bf a}_n)_{n\geq 0}$ being the dual basis associated to $(P_n)_{n\geq 0}$. Taking $n=0$ in the above expression we have 
\begin{align}\label{distri-011}
{\bf D}_{1/q,-\omega /q}\big((x-c){\bf u} \big)=\varphi(x){\bf u},
\end{align} 
where $\varphi$ is a polynomial of degree two given by
\begin{align*}
\varphi=-\frac{q}{C_1C_2}((a_0+b_0)C_1C_2 +(r_1+b_1B_0)C_2P_1+b_2C_1P_2).
\end{align*}
Taking $n=0$, $n=1$, and $n=2$ in \eqref{22} we get
\begin{align*}
a_0+b_0=0,\qquad c_1+b_1B_0=B_0-c,\qquad b_2=q+1 +\frac{(B_0-c)(qB_0-B_1+\omega)}{C_1}.
\end{align*}
Hence $C_2\varphi(x)=-q\,b_2(x-\lambda_+)(x-\lambda_{-})$, and the first part of the theorem follows. Assume that $\varphi(\omega +q\,c)=1$. Recall that (see \cite[(2.10)]{ACMP1})
\begin{align*}
{\bf D}_{1/q,-\omega/q}(f{\bf u})=D_{1/q,-\omega/q}f\,{\bf u} + f\left((x-\omega)/q\right){\bf D}_{1/q,-\omega/q}{\bf u} \qquad (f\in \mathcal{P}).
\end{align*}
Using this identity,  \eqref{distri-011} becomes
\begin{align*}
(x-qc-\omega){\bf D}_{1/q,-\omega/q}{\bf u}=q(\varphi(x)-1){\bf u}.
\end{align*}
Since $\varphi(\omega +qc)=1$, $x-qc-\omega$ and $q(\varphi(x)-1)$ have a common zero at $x=qc+\omega$, and therefore there exists a polynomial of degree one, $Q_1$, such that $q(\varphi(x)-1)=(x-qc-\omega)Q_1(x)$, which gives
\begin{align*}
(x-qc-\omega)\Big(({\bf D}_{1/q,-\omega/q}{\bf u})-Q_1(x){\bf u}\Big)=0,
\end{align*}
and so
\begin{align*}
 {\bf D}_{1/q,-\omega/q}{\bf u}= \frac{1}{(q^{-1}-1)rs }(x-r-s-\omega/(1-q)){\bf u},
\end{align*}
for some nonzero complex numbers $r$ and $s$, i.e. $Q_1(x)=1/((q^{-1}-1){rs}) (x-r-s-\omega/(1-q))$. This last equation is of type \eqref{NUL-Pearson2} with $\phi=1$ and 
$\psi(x)=1/((q^{-1}-1){rs}) (x-r-s-\omega/(1-q))$. We claim that $\mathbf{u}$ is regular. Indeed, by Theorem \ref{Dqw-main-Thm},  $c=1$, $b=0$, $a=0$, $d=1/((q^{-1}-1)rs)$, and $e=-(r+s+\omega/(1-q))/((q^{-1}-1)rs)$. Hence 
$$
d_n=q^{-n}/((q^{-1}-1)rs) \neq 0, \qquad \phi=1\neq 0.
$$
Moreover, by Theorem \ref{Dqw-main-Thm}, we get
\begin{align*}
B_n&=-q^{-1}\omega[n]_{q^{-1}}+\frac{[n]_{q^{-1}}e_{n-1}}{d_{2n-2}}-\frac{[n+1]_{q^{-1}} e_n}{d_{2n}} \\[7pt]
&= \omega/(1-q) +(r+s)q^n,\\[7pt]
C_{n+1}&=-\frac{q^{-n}[n+1]_{q^{-1}}d_{n-1}}{d_{2n-1}d_{2n+1}}=-rs(1-q^{n+1})q^n, 
\end{align*}
and finally $P_n(x)=s^nU_n ^{(r/s)} ((x-\omega/(1-q))/s\,|\,q)$, where $(U_n^{(a)}(\cdot\,|\,q))_{n \geq 0}$ are the Al-Salam-Carlitz polynomial (see \cite[Section 14.24]{K}).  Assume now that $\varphi(\omega +qc)\neq 1$. Hence $x-\omega-qc$ and $\varphi(x)-1$ are coprime. Using the same {\em argumentum ad absurdum} as in Theorem \ref{theo-2b}, we see that ${\bf u}$ is a ${\bf D}_{1/q,-\omega/q}$-semiclassical form of class one. Now, from \eqref{distri-011}, and using \eqref{condt}, we get
$$\lambda_+=c, \qquad \lambda_{-}=c-\left(\left(B_0-B_1 -\frac{c-B_0}{b_2C_1}C_2\right)^2 +4C_1\right)^{1/2}.$$
Then \eqref{distri-011} becomes
\begin{align*}
{\bf D}_{1/q,-\omega /q}\big((x-c){\bf u} \big)=-\frac{qb_2}{C_2}(x-c)(x-c+\Delta^{1/2}){\bf u},
\end{align*} 
where $\Delta= \left(B_0-B_1 -\dps\frac{c-B_0}{b_2C_1}C_2  \right)^2 +4C_1$ or, equivalently, 
\begin{align*}
{\bf D}_{1/q,-\omega/q}((x-c){\bf u})=\frac{1}{(q^{-1}-1)rs}(x-c)(x-r-s-\omega/(1-q)){\bf u},
\end{align*}
where $r$ and $s$ are nonzero complex numbers such that $(q-1)r s=C_2/b_2$ and $r+s=c-\omega/(1-q)-\Delta^{1/2}$. Define ${\bf v}= (x-c){\bf u}$. Hence
\begin{align*}
{\bf D}_{1/q,-\omega/q}{\bf v}= \frac{1}{(q^{-1}-1)r s}(x-r-s-\omega/(1-q)){\bf v}.
\end{align*}
As above, by Theorem \ref{Dqw-main-Thm},  
$$
d_n=\frac{q^{-n}}{(q^{-1}-1)rs} \neq 0, \qquad \phi=1\neq 0.
$$
Moreover, also by Theorem \ref{Dqw-main-Thm}, ${\bf v}$ is the linear form corresponding to the Al-Salam-Carlitz polynomials, and the theorem follows.
\end{proof}

The next proposition gives an explicit example of semiclassical orthogonal polynomials of class one satisfying \eqref{22}, which prevents the reader from making any conjectures related to classical polynomials when faced with a relation of type \eqref{Askey} after changing the standard derivative by Hahn's operator.
\begin{proposition}
Fix $\omega, q \in \mathbb{C}$  such that \eqref{wq} hold. Fix $a, b \in \mathbb{C}$ such that
 $$
-a\not=b,\qquad b\not=0,\qquad  a+(-1)^nb -(a+b)q^n \neq 0,
 $$
for all $n \in \mathbb{N}$. Let $(P_n)_{n\geq 0}$  be a sequence of orthogonal polynomials satisfying \eqref{TTRR} with
\begin{align}\label{coef-new-ops-Hahn-case}
B_n=\frac{\omega}{1-q},\qquad C_{n}=(a+b)\left(\frac{a+(-1)^nb}{a+b}-q^n  \right)q^n.
\end{align}
Then $(P_n)_{n\geq 0}$ is a $D_{1/q,-\omega/q}$-semiclassical of class one and its corresponding linear form $\mathbf{u}\in \mathcal{P}'$ satisfies 
\begin{align*}
&{\bf D}_{1/q,-\omega/q}\left(\left(x-\frac{\omega}{1-q}\right)\mathbf{u}\right)\\[7pt]
&\quad =\frac{1}{(a+b)(q-1)}\left(\frac{1}{q}\left(x-\frac{\omega}{1-q}\right)^2 +b-a+(a+b)q\right){\bf u}.
\end{align*}
\end{proposition}

\begin{proof}
We claim that $(P_n)_{n\geq 0}$ satisfies \eqref{22} with
$$
c=\frac{\omega}{1-q}, \qquad a_n=\frac{1+(-1)^{n+1}}{(1-q)(a+b)}\, b, \qquad b_n=[n]_q -a_n, \qquad c_n=-c b_n.
$$
Indeed, the proof is by induction on $n$. For $n=1$, LHS of \eqref{22} gives 
$$(x-c)D_{q,\omega}P_1(x)=x-c,$$ while RHS gives $$ a_1P_1(x) +b_1(x-c)P_0(x)=(a_1+b_1)(x-c)=x-c.$$  Assuming \eqref{22} to hold, with $k$ instead of $n$, for $k=1, 2, \dots, n$, we will prove it for $k=n+1$. Apply $(x-c)D_{q,\omega}$ to \eqref{TTRR} to obtain
\begin{align*}
(x-c)D_{q,\omega}\big((x-c)P_{n}(x)\big)=(x-c)D_{q,\omega}\big( P_{n+1}(x) +C_nP_{n-1}(x)\big).
\end{align*}
Using the identity (see \cite[(2.9)]{ACMP1})
$$D_{q,\omega}(fg)=g(qx+\omega)D_{q,\omega}f +fD_{q,\omega}g \qquad (f,g\in \mathcal{P}),$$
we get
\begin{align*}
(x-c)D_{q,\omega}P_{n+1}(x)&=q(x-c)^2D_{q,\omega}P_{n}(x) +(x-c)P_{n}(x) -(x-c)C_nD_{q,\omega}P_{n-1}(x).
\end{align*}
Now using successively \eqref{22}, with $k$ instead of $n$, for $k=n$ and $k=n-1$ and \eqref{TTRR} it follows that
\begin{align*}
(x-c)D_{q,\omega}P_{n+1}(x)&=q(x-c)(a_nP_{n}(x)+b_n(x-c)P_{n-1}(x)) \\[7pt]
&\quad +(x-c)P_{n}(x)-(a_{n-1}P_{n-1}(x)+b_{n-1}(x-c)P_{n-2}(x))C_n\\[7pt]
&=qa_n (P_{n+1}(x) +C_nP_{n-1}(x)) +qb_n(x-c)(P_{n}(x) +C_{n-1}P_{n-2}(x))\\[7pt]
&\quad +(x-c)P_{n}(x)-(a_{n-1}P_{n-1}(x)+b_{n-1}(x-c)P_{n-2}(x))C_n\\[7pt]
&=qa_nP_{n+1}(x)+(qa_n-a_{n-1})C_nP_{n-1}(x) +(1+qb_n)(x-c)P_{n}(x) \\[7pt]
&\quad +(qb_nC_{n-1}-b_{n-1}C_n)(x-c)P_{n-2}(x)\\[7pt]
&=a_{n-1}P_{n+1}(x) +(1-a_{n-1}+q(a_n+b_n))(x-c)P_{n}(x)\\[7pt]
&\quad +(qb_nC_{n-1}-b_{n-1}C_n)(x-c)P_{n-2}(x).
\end{align*}
The reader should convince himself that the following relations hold: $a_{n-1}=a_{n+1}$, $1-a_{n-1}+q(a_n+b_n)=b_{n+1}$, and $qb_nC_{n-1}-b_{n-1}C_n=0$. Thus \eqref{22} holds for $n+1$, and our claim follows. Note also that 
\begin{align*}
\qquad b_2=q+1, \qquad C_1=(a-b-(a+b)q)q,\qquad  C_2=(a+b)(1-q^2)q^2.
\end{align*}
Under the notation of Theorem \ref{theo-22b}, we get
\begin{align*}
\lambda_\pm&=c\pm (a-b-(a+b)q)^{1/2}q^{1/2},
\end{align*}
and so
\begin{align*}
&q(\omega +qc-\lambda_+)(\omega +qc-\lambda_{-})\\[7pt]
&\quad =(b-a+(a+b)q)q^2\not=(-b-a+(a+b)q)q^2=-\frac{C_2}{b_2},
\end{align*}
because $\omega+q\,c=c$. Thus, from Theorem \ref{theo-22b}, the result follows.
\end{proof}

If we replace, in Theorem  \ref{main-theorem}, the Askey-Wilson operator by the Hahn operator, then ${\bf S}_q$ becomes the identity, as in the case of the standard derivative, and Theorem \ref{main-theorem} $i)$ reduces to ${\bf D}_{q,\omega}(\rho {\bf u})=\psi {\bf u}$. In this context, an analogue of Theorem \ref{main-theorem} appears in Smaili's PhD thesis under the supervision of Maroni (see \cite[Theorem 1.1]{S87}).

\section*{Acknowledgements}
The authors thank T. H. Koornwinder for several helpful comments. The authors also thank P. Maroni for his comments, and for offering them an original printed copy of \cite{M87} during his last visit to Portugal. This work was supported by the Centre for Mathematics of the University of Coimbra-UIDB/00324/2020, funded by the Portuguese Government through FCT/ MCTES.  The first author thanks CMUP, University of Porto, for their support and hospitality. The second author also thanks CMUP for their support. CMUP is a member of LASI, which is financed by national funds through FCT, under the projects with reference UIDB/00144/2020 and UIDP/00144/2020.

 \end{document}